\documentclass [a4paper,twoside,12pt]{article}

\usepackage[utf8]{inputenc}
\usepackage{vmargin,graphicx,theorem}
\usepackage[english]{babel}
\usepackage{enumerate}
\usepackage{color}

\RequirePackage[colorlinks,linkcolor=blue,citecolor=blue,urlcolor=blue]{hyperref}
\usepackage{amssymb, amsmath, amsfonts, latexsym}
\usepackage{enumerate,graphicx}
\usepackage{pst-node}

\setpapersize[portrait]{A4}
\setmarginsrb{1.5cm}{1cm}{1.5cm}{2.5cm}{1.5cm}{0cm}{0.5cm}{2cm}

\newtheorem{ethm}{Theorem}
\newtheorem{ecor}[ethm]{Corollary}
\newtheorem{eprop}[ethm]{Proposition}
\newtheorem{elem}[ethm]{Lemma}
\newtheorem{edefi}[ethm]{Definition}
\newtheorem{erem}[ethm]{Remark}

\newcommand{\proofend}{~$\rhd$}
\newcommand{\proofbegin}{~$\lhd$}
\newenvironment{proof}
               {\noindent {\emph{\textbf{Proof}}}\\\proofbegin~}
              {\proofend\\}

\newcommand{\A}{\mathcal{A}}

\newcommand{\R}{\mathbb{R}}
\newcommand\XX{\mathbb{R}^n}
\newcommand\dx{\mathcal{L}}

\newcommand{\PO}{\mathcal P(\Omega)}
\newcommand{\PX}{\mathcal P(\XX)}

\newcommand{\gd}{\Gamma_2}

\newcommand{\e}{\varepsilon}

\title{Convexity and regularity properties for entropic interpolations}
\date{\today}

\author{Luigia Ripani\thanks{Institut Camille Jordan, Umr Cnrs 5208, Université Claude Bernard Lyon 1. ripani@math.univ-lyon1.fr} }

\begin{document}

\maketitle

\abstract{In this paper we prove a convexity property of the relative entropy along entropic interpolations (solutions of the Schr\"odinger problem), and a regularity property of the entropic cost along the heat flow. Then we derive a dimensional EVI inequality and a contraction property for the entropic cost along the heat flow. As a consequence, we recover the equivalent results in the Wasserstein space, proved by Erbar, Kuwada and Sturm}\\

\noindent
\textbf{R\'esumé} Dans cet article nous d\'emontrons une propri\'et\'e de convexit\'e de l'entropie relative le long des interpolations entropiques (solutions du problème de Schr\"odinger), et une propri\'et\'e de r\'egularit\'e du co\^ut entropique le long du flot de la chaleur. Ensuite, nous en d\'eduisons une in\'egalit\'e EVI dimensionnelle et une propri\'et\'e de contraction pour le co\^ut entropique le long du flot de la chaleur. En cons\'equence, nous retrouvons les r\'esultats \'equivalents dans l'espace de Wasserstein, d\'emontrés par Erbar, Kuwada et Sturm.

\noindent
\medskip

\medskip


\section{Introduction}\label{sec-intro-N}

Convexity of the entropy along evolutionary equations is a powerful tool to prove regularity properties, asymptotic behavior, etc. We extend and compare some main and fruitful results around convexity of the entropy in the Wasserstein space, to the context of the Schr\"odinger problem.


For simplicity, results are presented in $\R^n$ associated with the Lebesgue measure $\dx$. Generalization can be stated in the context of a  $n$-dimensional Riemannian manifold $(M^n,g)$. 

We consider the relative entropy functional, loosely defined for any couple of positive measures $\mu$, $\nu$ on $\R^n$ as
\begin{equation*}
H(\mu|\nu)=\int \log\left(\frac{d\mu}{d\nu}\right)d\mu,
\end{equation*}
whenever the integral is meaningful.

\subsubsection*{Context of the Wasserstein space} 
The quadratic Monge-Kantorovich distance between two probability measures $\mu_0,\mu_1 \in \mathcal P_2(\XX)= \{ \mu\in \mathcal P(\XX) : \int |x|^2 d\mu<\infty\}$, is defined as
$$
W_2^2(\mu_0,\mu_1) := \inf_\pi \left\{\int_{\XX\times\XX} |x-y|^2\pi(dxdy)\right\},
$$ 
where the infimum is running over all the couplings $\pi$ of $\mu_0$ and $\mu_1$, namely, all the probability measures $\pi\in \mathcal P(\XX\times\XX) $ with marginals $\mu_0$ and $\mu_1$, that is for any bounded measurable functions $\varphi$ and $\psi$ on $\XX$, 
$
\int [\varphi(x)+\psi(y)]\pi(dxdy) = \int \varphi \,d\mu_0+\int \psi \,d\mu_1.
$

The space $(\mathcal P_2(\XX), W_2)$ is geodesic. This means that for any couple $\mu_0, \mu_1 \in P_2(\XX)$ there is a path $(\mu_s^{MC})_{s\in[0,1]}$ in $\mathcal P_2(\XX)$  such that, 
$$
W_2(\mu_s^{MC},\mu_t^{MC}) = |t-s|W_2(\mu_0,\mu_1) \quad \forall \;s,t \in [0,1].
$$

$\bullet$ The first result in this context is the convexity of the entropy along geodesics, i.e. if $(\mu_s^{MC})_{s\in[0,1]}$ is a geodesic  in $\mathcal P_2(\XX)$, then the map 
$$
[0,1]\ni s\mapsto H(\mu_s^{MC}|\dx)
$$
is convex. 
The entropy is then \emph{displacement convex}  in the sense introduced by McCann in~\cite{McC97}. This was a breakthrough,  a starting point of the Lott-Sturm-Villani theory, who defined the (positive) curvature in a {\it metric measure space} (mms space), see~\cite{LV09, ST06}. This is usually noted the  $CD(0,\infty)$ condition.

$\bullet$ Taking into account the dimension, the main progress is proposed by Erbar-Kuwada-Sturm in~\cite{EKS} who proved the stronger result, under the same assumption, that the map 
$$
[0,1]\ni s\mapsto N(\mu_s^{MC}):=\exp\left(-\frac{1}{n}H(\mu_s^{MC}|\dx)\right), 
$$
is concave. This condition can be used to define the $CD(0,n)$ condition in a mms space. This result has two main applications. 

$\rightarrow$ First a dimensional \emph{evolution variational inequality} (EVI) for the quadratic Monge-Kantorovich  distance, that writes for all $t\geq 0$ as,
\begin{equation}\label{evi-w_2}
\frac{d}{dt}^+ W_2^2(T_tu\,\dx,v\,\dx) \leq n\left(1-e^{-\frac{1}{n}\left[H(v|\dx)-H(T_tu|\dx)\right]}\right),
\end{equation}
for any $u,v \in \mathcal P_2(\XX)$ where $(T_t)_{t\geq0}$ is the heat semi-group in $\R^n$. This inequality is actually equivalent to say that the heat flow is the gradient flow associated to the entropy functional \cite{AGS05}.

$\rightarrow$ Then, it provides a proof of the dimensional contraction with respect to the $W_2$ distance, for any  $u,v \in \mathcal P_2(\XX)$  and any $\tau>0$, 
\begin{equation}
\label{eq-contraction-W}
W_2^2(T_\tau u\dx,T_\tau v\dx)\leq W_2^2(u,v)  -4n\int_0^\tau \sinh^2\left(\frac{H(T_tu|\dx)-H(T_tv|\dx)}{2n}\right) dt,
\end{equation}
result proved in~\cite{BGG, BGGK}.

\subsubsection*{Context of the Schr\"odinger problem} 
It will be properly defined in Section~\ref{setting}. Roughly speaking, the \emph{entropic cost} associated to the Schr\"odinger problem, is, up to a constant term, the minimization problem, 
\begin{equation}\label{entropic cost-intro}
\A(\mu_0,\mu_1) = \inf \{ H(P|R)\;;\; P\in\PO \;\textrm{s.t.}\; P_0=\mu_0 \;\textrm{and}\; P_1=\mu_1\} \; \in (-\infty,\infty],
\end{equation}
for a fixed reference measure $R \in \mathcal{M}_+(C([0,1],\XX))$, that we can  consider to be the Brownian motion on $\XX$ with reversing measure the Lebesgue measure, and marginals $\mu_0, \mu_1$ in some restriction of the set $\mathcal P_2$, that will be defined later at Section~\ref{setting}. If we denote $\hat P\in\PO$ the minimizer of~\eqref{entropic cost-intro}, the \emph{entropic interpolation} is defined as its marginal flow, i.e.
$$
\mu_t := \hat P_t=(X_t)_\#\hat P \in P(\XX), \quad \textrm{for any} \; 0\leq t\leq 1.
$$

$\bullet$ The first result on the subject is due to L\'eonard~\cite{Leo12d} who proved that for any entropic interpolation $(\mu_s)_{s\in[0,1]}$, the map 
$$
[0,1]\ni s\mapsto H(\mu_s|\dx), 
$$
is convex, where $\dx$ denotes the Lebesgue measure on $\XX$.

One particular entropic geodesic is the map $\mu_s=T_sf$, $s\in[0,1]$,  where $f$ is a smooth probability density. This is one of the starting point of the Bakry-\'Emery-Ledoux theory to prove regularity, asymptotic behaviour etc. of diffusion Markov generators (see~\cite{BGL14}). 

In this context, adding the dimension, again it is more convenient to look at the exponential entropy. In 1985, Costa~\cite{Cos85} proved that the map 
\begin{equation}\label{intro-costa}
\R^{+}\ni s\mapsto N(T_sf)= \exp\left(-\frac{2}{n}H(T_sf|\dx)\right)
\end{equation}
is concave.  This is an important result in information theory, and it is also useful to prove some functional inequalities, for instance the dimensional log-Sobolev inequality as it is reported in~\cite[Ch. 10]{ABC00}.

All these results have their counterparts in the more general case of Ricci curvature bounded from below by some $\kappa \in \R$. We refer to references such as \cite{Con17} for  general cases, Riemannian manifolds or mms spaces.

\bigskip

The aim of the paper is to complete the picture concerning convexity and regularity for the Schr\"odinger problem. First we prove a Costa type result for the entropic interpolation, that is for any entropic interpolation $(\mu_s)_{s\in[0,1]}$, the map 
$$
[0,1] \ni s \mapsto \exp\left(-\frac{1}{n}H(\mu_s|\dx)\right),
$$
is concave, cf. Theorem~\ref{costa}. Note the absence of the factor $2$ with respect to \eqref{intro-costa}. The relation between these two expressions will be done at Remark~\ref{rem-costa}. Secondly, we prove that for any $u$ and $v$, probability densities  in some space that will be specified later, the map 
$$
\R^+\ni t\mapsto \mathcal A(T_tu\,\dx, v\,\dx), 
$$
is differentiable and, 
$$
\frac{d}{dt}\Big|_{t=0}\mathcal A(T_tu\,\dx, v\,\dx)=-\frac{1}{2}\frac{d}{ds}\Big|_{s=1} H(\mu_s|\dx),
$$
where $(\mu_s)_{s\in[0,1]}$ is the entropic interpolation between $u$ and $v$, cf. Theorem~\ref{derivarive}. 

From these two results, one can deduce an EVI inequality for the entropic cost (Corollary~\ref{evi}) 
\begin{equation*}\label{evi-intro}
\frac{d}{d t} \mathcal A(f\,\dx,T_t g\,\dx) \leq \frac{n}{2}\left(1-e^{-\frac{1}{n}[H(f|\dx) - H(T_t g|\dx)]}\right),
\end{equation*}
for any $t>0$ and a dimensional contraction inequality (Corollary~\ref{cor-contr}) of the entropic cost along the heat flow,
\begin{eqnarray*}
\A(T_\tau u\,\dx,T_\tau v\,\dx)-\A(u\,\dx,v\,\dx) \leq -n\int_0^\tau \sinh^2\left(\frac{H(T_tu|\dx)-H(T_tv|\dx)}{2n}\right) dt.
\end{eqnarray*}

In conclusion, this approach provides an easy and rigorous proof of the analogous results for the Wasserstein distance. In fact, if we introduce a parameter $\varepsilon>0$ in the entropic cost as follows, 
\begin{equation*}
\A^\varepsilon(\mu_0,\mu_1) = \inf \{\varepsilon H(P|R^\varepsilon);\; P\in\PO \;\textrm{s.t.}\; P_0=\mu_0 \;\textrm{and}\; P_1=\mu_1\}-\frac{\e}{2}[H(\mu_0|m)+H(\mu_1|m)]
\end{equation*}
where $R^\varepsilon$ is the reference path measure associated with the generator $L^{\varepsilon}=\varepsilon\Delta/2$. As proved in~\cite{Leo12a, Mika04} we have the convergence property,  
\begin{equation*}
\lim_{\e\to 0} \A^\e(\mu_0,\mu_1) =\frac{W^2_2(\mu_0,\mu_1)}{2}.
\end{equation*}
Thanks to this result, it is an immediate consequence to derive~\eqref{evi-w_2} and~\eqref{eq-contraction-W} as limits respectively of EVI and the contraction for the $\varepsilon$-entropic cost. \\
\newline
The paper is organized as follows. In the next section, we give in full details the setting of the Schr\"odinger problem. In Section~\ref{sec-results} we present and prove our main results. First, in Theorem~\ref{costa} we prove concavity of the exponential entropy along entropic interpolations. Then, in Theorem~\ref{derivarive} we prove the regularity property of the entropic cost. We derive from these two results the EVI inequality for the entropic cost (Corollary~\ref{evi}), the dimensional contraction along the heat flow (Corollary~\ref{cor-contr}) and an integral form of EVI (Proposition~\ref{prop-integral}). Finally, we deduce contraction in Wasserstein distance (Remark~\ref{rem-contr-w}) and the classical EVI for the Wasserstein distance (Corollary~\ref{cor-w}).

\section{Setting}\label{setting}

In this section we fix notations and recall some definition and property of the main objects of our framework.  \\
First we extend the definition of relative entropy to measures that are not necessary finite. Let $r\in\mathcal M_+(Y)$ be $\sigma$-finite. Then, it exists at least a measurable non negative function $W$ such that $z_W= \int e^{-W}dr<\infty$. By defining the probability measure $dr^W=e^{-W}dr/z_W$, we can write, $H(p|r)=H(p|r^W)-\int Wdp-\log  z_W$. And it is well defined for any $p\in\mathcal P(Y)$ such that $\int Wdp<\infty$. Therefore the \emph{relative entropy} of a probability measure $p\in \mathcal{P}(Y)$ such that $\int W dp<\infty$, with respect to a positive $\sigma$-measure $r\in \mathcal{M}_+(Y)$ is  defined by,
\begin{equation}\label{entropy}
(-\infty, \infty] \ni H(p|r)=\left\{\begin{array}{ll}
\displaystyle\int \log\displaystyle \frac{dp}{dr}\,dp& \textrm{if} \; p\ll r\\ 
+\infty & \textrm{otherwise}.
\end{array}\right.
\end{equation}
For more details about relative entropy, conditional expectation and disintegration for unbounded measures, see~\cite{Leo12b}.\\
A \emph{reference path measure} $R\in \mathcal{M}_+(\Omega)$, is a positive measure on the set $\Omega=C([0,1],Y)$.
We fix the state space $Y=\XX$, equipped with its Borel $\sigma$-field, and the path space $\Omega$, with the canonical $\sigma$-field $\sigma(X_t; \;0\leq t\leq 1)$ generated by the canonical process for all $t \in [0,1]$,
$$
X_t(\omega):=\omega_t \in \XX, \quad \omega=(\omega_s)_{0\leq s\leq1} \in \Omega.
$$
Moreover, for any measure $Q\in \mathcal{M}_+(\Omega)$ and any $t\in [0,1]$, we denote
$$
Q_t(\cdot) := Q(X_t\in\cdot)=(X_t)_\#Q \in \mathcal{M}_+(\XX).
$$
The reversing and Lebesgue measure coincide, hence without ambiguity by abuse of notation we will not distinguish between density functions and the measures. 

\subsubsection*{The reference path measure}\label{reference}
In the sequel, as reference path measure, we consider the reversible Brownian motion $R$ on the state space $Y=\XX$, with generator $L=\Delta/2$ and initial condition $R_0(dx):=\dx$ the Lebesgue measure on $\XX$. Note that it is an unbounded measure since $R$ has the same mass as $R_0$ (See~\cite{Leo12e}). In particular, we denote $(T_t)_{t\geq0}$ the Markov semigroup associated to the generator $L$, which is defined for any bounded measurable function $f$ and any $t\geq 0$ and $x\in\XX$ by, 

\begin{equation}\label{heat}
T_tf(x)= \int f(y) \frac{e^{-|x-y|^2/2t}}{(2\pi t)^{n/2}} \,dy.
\end{equation}
Reversibility will play a crucial role at different points of our proofs. We recall that it is equivalent to say that for any couple of functions $f, g \in C_c^\infty(\XX)$, 

$$\int fT_tg\, dx=\int gT_tf\, dx.$$
In other words, let us define the time reversal mapping $X^*_t:=X_{1-t},\; 0\leq t\leq 1$ and $R^*=(X^*)_\#R$ the time-reversed of $R$. Reversibility means that $R^*=R$. Moreover it implies that $R$ is $\dx$-stationary, that is, $R_t=\dx$, for all $0\leq t\leq 1$, and this is equivalent to say that the generator is symmetric, that is for $f, g \in C_c^\infty(\XX)$,
$$
\int fLg \,dx=\int gLf \,dx.
$$
It can be easily verified that the \emph{carré du champ} operator, defined for any couple of functions $f,g \in C_c^\infty(\XX)$, by $\Gamma(f,g)=[L(fg)-fLg-gLf]/2$, when it is associated to $L=\Delta/2$, is given by $\Gamma(f,g)=\nabla f\cdot \nabla g/2$. It satisfies the integration by parts formula, 
\begin{equation}\label{ipp}
\int f\Delta g\,dx=-\int \nabla f\cdot \nabla g\,dx=-2\int \Gamma(f,g) \,dx.
\end{equation}
Moreover the iterated \emph{carré du champ} operator $\gd$, defined for $f\in C_c^\infty(\XX)$ by, 
$$
\Gamma_2(f):= \frac{1}{2}L(|\nabla f|^2)-\nabla f\cdot\nabla Lf 
$$
when the generator $L=\Delta/2$, is $\gd(f)=||\textrm{Hess} f||_2^2/4$ and by the Cauchy-Schwarz inequality, it yields to, 

\begin{equation}\label{cd}
\gd(f) \geq \frac{1}{4n} (\Delta f)^2=\frac{1}{n}(Lf)^2.
\end{equation}
Inequality~\eqref{cd} is known under the name of $CD(0,n)$ condition, introduced by Bakry and Emery in~\cite{BE85}, see also~\cite{BGL14}.

\subsubsection*{Entropic cost $\A$}

In order to define the entropic cost, we need to fix, in addition to a reference path measure, two marginal probability measures on the state space $\XX$. In particular we assume that the marginals, 
\begin{equation}\label{hyp-constr}
\mu_0, \mu_1\in  \Pi:= \left\{\mu \in P(\XX) \;:\; H(\mu|\dx)<\infty, \int |x|^2\,d\mu <\infty\right\}.
  \end{equation}
Under this assumption, the \emph{entropic cost} is defined as 
\begin{equation}\label{entropic cost}
\A(\mu_0,\mu_1) := \inf \{ H(P|R);\; P\in\PO \;\textrm{s.t.}\; P_0=\mu_0 \;\textrm{and}\; P_1=\mu_1\} -\frac{1}{2}[H(\mu_0|\dx)+H(\mu_1|\dx)]
\end{equation}
In order to show that the entropic cost is well defined and finite, it is more convenient to use the equivalent static definition. Let us consider the joint law of the initial and final position of the reversible Brownian motion $R$, that is, 
$$
R_{01}(dxdy)=\frac{e^{-|x-y|^2/2}}{(2\pi)^{n/2}}dxdy.
$$
It is shown in~\cite[Prop. 2.3]{Leo12e} that the entropic cost can be defined equivalently as,
$$
\A(\mu_0,\mu_1) = \inf \{ H(\pi|R_{01});\; \pi\in\mathcal P(\XX\times\XX) \;\textrm{s.t.}\; \pi_0=\mu_0, \; \pi_1=\mu_1\}-\frac{1}{2}[H(\mu_0|\dx)+H(\mu_1|\dx)]
$$
where $\pi_0:= \pi(\cdot\times\XX)$, and $\pi_1:= \pi(\XX\times\cdot)$. The assumption on the marginals $\mu_0,\mu_1$ to have second order moment finite in \eqref{hyp-constr}, implies that the relative entropy with respect to $R_{01}$ is bounded from below. To see this, it is enough to choose $W(x,y)=|x|^2+|y|^2$ in the definition of relative entropy \eqref{entropy}. Moreover, the assumption of finite relative entropy in \eqref{hyp-constr} together with the fact that $R_{01}(dxdy)\geq e^{-|x|^2-|y|^2}dxdy$, makes sure that $H(\mu_0\otimes\mu_1|R_{01})<\infty$, therefore $\A(\mu_0,\mu_1)$ is bounded also from above (See~\cite[Lemma 2.a, Prop. 2.5]{Leo12e} for more details and the general case).\\
However, in order to enunciate rigorously our results we need some more restrictive assumption on the marginals. In particular, we will assume $\mu_0, \mu_1$ to be smooth and compactly supported probability measures. As a consequence of this assumption we have that $\mu_0,\mu_1\in \Pi$. A stronger result in Theorem~\ref{derivarive} is stated under stronger assumptions that will be specified later at \eqref{hyp-pis}.

 
\begin{erem}\label{remark}
Note that unlike the $W_2^2$, $\A$ is not the square of a distance. Though, by the reversibility of the reference measure $R$, $\A$ is symmetric, that is, for all suitable $\mu_0,\mu_1$, $$\A(\mu_0,\mu_1)=\A(\mu_1,\mu_0).$$ Indeed, since the time reversing mapping $X^*$ (defined above) is one-to-one, it holds $H(P|R)=H(P^*|R^*)$. Moreover,  since $R$ is reversible, it implies that, $H(P|R)=H(P^*|R)$. Thus, if $P$ is a minimizer in~\eqref{entropic cost}, then $P^*$ is the minimizer of the same Schr\"odinger problem with switched marginals.  
\end{erem}
We introduce here a fluctuation parameter, that will allow us to link the entropic cost to the quadratic Wasserstein distance. Let $R^\e$ be the law of the reversible Brownian motion with infinitesimal generator 
$$
L^\varepsilon=\frac{\varepsilon}{2}\Delta, \quad \textrm{for any} \quad  \varepsilon>0.
$$
Note that it doesn't change the dynamics, but it corresponds to a simple dilatation in time. We define the $\e$-entropic cost as, 
\begin{equation}\label{eps-ec}
\A^\e(\mu_0,\mu_1) = \inf \{ \e H(P|R^\e);\; P\in\PO \;\textrm{s.t.}\; P_0=\mu_0, \; P_1=\mu_1\}-\frac{\e}{2}[H(\mu_0|\dx)+H(\mu_0|\dx)].
\end{equation}
Note the rescaling factor $\e$ in front of the entropy, in order $\A^\e$ not to explode to infinity in the limit when $\e$ vanishes. It is shown in~\cite{Leo12a} via large deviation arguments, that the $\e$-entropic cost is a regular approximation of the square of the quadratic Wasserstein distance, namely, 
\begin{equation}\label{gamma-lim}
\lim_{\e\to 0} \A^\e(\mu_0,\mu_1) =\frac{W^2_2(\mu_0,\mu_1)}{2}.
\end{equation}
We only use the $\varepsilon$-entropic cost $\A^\varepsilon$ to recover the classic results for $W_2$ at Remark~\ref{rem-contr-w} and Corollary~\ref{cor-w} at the end of the article. 
\subsubsection*{Entropic interpolations}
We assume from now on, that the reference path measure $R$ is associated to the heat semigroup as introduced at the beginning of this section, and $\mu_0, \mu_1 \in C_c^\infty(\XX)$. Under these assumptions, by~\cite[Thm. 2.12]{Leo12e}, a unique minimizer exists, is called \emph{entropic bridge} and is characterized by the formula, 
\begin{equation}\label{minimizer}
\hat P=f(X_0)g(X_1) R \; \in \PO,
\end{equation}
where $f, g$ positive, are the unique solutions of the Schr\"odinger system (cf.~\cite{Foe85}), 

\begin{equation}\label{sch-system}
\left\{\begin{array}{cc}
     \mu_0 &=fT_1g \\
     \mu_1 &=gT_1f. 
\end{array}\right.
\end{equation}
Here $T_1$ is the heat semigroup~\eqref{heat} at time $t=1$, and again by abuse of notation we denote by  $\mu_i$ also the density of the probability measure $\mu_i$ with respect to the Lebesgue measure, for $i=0,1$. 

\begin{erem}
As proved in \cite[Thm.~3.1]{GT}, and in \cite{Tam} in a non compact setting, the assumption on the marginal measures to have densities with respect to the Lebesgue measure in $C_c^\infty(\XX)$, makes sure that the functions $f,g$ solution of \eqref{sch-system} are positive and $L^\infty(\XX)$ hence $C_c^\infty(\XX)$. Indeed, provided that $f,g\in L^\infty(\XX)$ then $T_tf, T_{1-t}g\in C^\infty(\XX)$ for any $0\leq t\leq 1$. Thus, 
$$
f=\frac{\mu_0}{T_1g}\in C_c^\infty(\XX)
$$
since $T_1g$ is smooth and strictly positive and $\mu_0\in C_c^\infty(\XX)$. The same argument is valid for $g$.
\end{erem}

\begin{edefi}[Entropic interpolation]
The $R$-entropic interpolation between $\mu_0$ and $\mu_1$ is defined as the marginal flow of the minimizer~\eqref{minimizer}, that is $\mu_t := \hat P_t=(X_t)_\# \hat P \in P(\XX)$ for $0\leq t\leq 1$. In particular, it is characterized by the formula, 
$$
d\mu_t=T_tfT_{1-t}g\dx, \quad \forall \;  0\leq t\leq 1,
$$
\end{edefi}
or equivalently, 
$$
d\mu_t=e^{\displaystyle\varphi_t+\psi_t}\dx, \quad \forall \;  0\leq t\leq 1,
$$
where for any $t\in [0,1]$, $\varphi_t:= \log T_tf$ and $\psi_t:=\log T_{1-t}g$ with $f,g$ solutions of~\eqref{sch-system} and $(T_t)_{t\geq0}$ the heat semigroup~\eqref{heat}. Note that since by hypothesis $f,g$ are $C_c^\infty(\XX)$, then $\mu_t$ is in $C^\infty(\XX)$ for all $0\leq t\leq 1$.\\
The two functions $\varphi_t$ and $\psi_t$ satisfy respectively the Hamilton-Jacobi-Bellman equations, 
\begin{equation}\label{psi+phi}
    \begin{array}{cc}
    \left\{
        \begin{array}{ll}
             \displaystyle \partial_t\varphi_t -\frac{\Delta\varphi_t}{2} -\frac{|\nabla\varphi_t|^2}{2}=0,  &  0\leq t\leq 1, \\[2ex]
            \varphi_0=\log f  & t=0
        \end{array}
    \right.
             &  \; 
    \left\{
        \begin{array}{ll}
           \displaystyle \partial_t\psi_t +\frac{\Delta\psi_t}{2}+\frac{|\nabla\psi_t|^2}{2}=0,  & 0\leq t\leq 1, \\[2ex]
            \psi_1=\log g,  &  t=1.
        \end{array}
    \right.
    \end{array}
\end{equation}
In analogy to the Kantorovich potentials, $\varphi_0$ and $\psi_1$ are often referred to as \emph{Schr\"odinger potentials}. By adding the PDEs for $\varphi$ and $\psi$ in~\eqref{psi+phi}, we deduce that the entropic interpolation is a smooth solution of the transport equation, 

\begin{equation}\label{transport}
\left\{\begin{array}{ll}
    \partial_t\mu_t +\nabla\cdot(\mu_t\nabla\theta_t) =0, & \forall \,t\in(0,1]    \\
     \mu_0 =\mu_0, & t=0, 
\end{array} \right.
\end{equation}
where $\nabla\theta_t=\nabla(\psi_t-\varphi_t)/2$. We briefly recall here the definitions of forward, backward, osmotic and current velocity introduced by Nelson in~\cite{Nel67}, and how they are related,

\begin{equation*}
    \begin{array}{cc}
    \left\{
        \begin{array}{ll}
            v^{cu}_t:=&\nabla\theta_t= \frac{\nabla\psi_t-\nabla\varphi_t}{2}    \\[2ex]
            v^{os}_t:=&\frac{1}{2}\nabla\log\mu_t= \nabla\psi_t+\nabla\varphi_t
        \end{array}
    \right.
             &  \qquad 
    \left\{
        \begin{array}{ll}
           \overrightarrow{v}_t:= \nabla \psi_t= \frac{1}{2}\nabla\log\mu_t+\nabla\theta_t \\[2ex]
           \overleftarrow{v}_t:= \nabla \varphi_t = -\nabla\theta_t+\frac{1}{2}\nabla\log\mu_t.
        \end{array}
    \right.
  
    \end{array}
\end{equation*}
Moreover, the definition of $P^*$, the time reversal of the minimizer in~\eqref{entropic cost}, implies that, 
$\mu^*_t=\mu_{1-t}=T_{1-t}fT_tg$, for any $0\leq t\leq 1$, and it establishes the following relation between the backward and the forward velocities respectively associated to $\mu_t$ and its time reversal $\mu_t^*$, 

\begin{equation}\label{vel-for-back}
\overleftarrow{v}_t(x)=\overrightarrow{v}^*_{1-t}(x), \qquad 0\leq t\leq 1, \,x\in \XX.
\end{equation}

\subsubsection*{Dual and Benamou-Brenier formulations}\label{sec-bb}
Finally, we recall two equivalent formulations of the entropic cost, that will be crucial in the proof of our main results. First, the dual formulation, in analogy with the Kantorovich formulation for the Monge problem~(\cite{MT06}, \cite[Section 4]{GLR}).  

\begin{ethm}[Dual Kantorovich formulation]\label{thm-dual}
For $\mu_0,\mu_1\in  C_c^\infty(\XX)$, then
\begin{equation}\label{dual}
\A(\mu_0,\mu_1)= \sup_{\psi\in C_b(\XX)} \left\{ \int \psi\, d\mu_1 - \int Q_1\psi \,d\mu_0 \right\}+\frac{1}{2}[H(\mu_0|\dx)-H(\mu_1|\dx)].
\end{equation}
Here, $Q_1\psi:=\log T_1e^{\psi}$, where $T_1$ is the heat semigroup at time $t=1$. 
The supremum is achieved by the Schr\"odinger potential $\psi_1=\log g$ that appears in \eqref{psi+phi}.
\end{ethm}
Then, the Benamou-Brenier formulation for the entropic cost, in analogy with the one for the Wasserstein distance proved in~\cite{BB00}. In the case of the entropic cost this formulation has been proved for the Brownian motion in~\cite{CGP14} and for a general Kolmogorov semigroup in~\cite[Section 5]{GLR}.

\begin{ethm}[Benamou-Brenier formulation]\label{bb}
Let $R$ be the Brownian motion on $\XX$, $\mu_0,\mu_1\in C_c^\infty(\XX)$, then 

$$
\A(\mu_0,\mu_1) = \frac{1}{2} \inf\int_0^1\int_{\XX}\left(|v_t(z)|^2+\frac{1}{4}|\nabla\log\mu_t(z)|^2\right)\mu_t(z)\,dzdt,
$$
where the infimum runs over all the couples $(\mu_t,v_t)_{0\leq t\leq1}$ such that

\begin{equation}\label{hyp-bb}
\left\{\begin{array}{ll}
    \mu_t \in P(\XX), \; \forall \; 0\leq t\leq 1, \\
    \mu_{t=0}= \mu_0, \; \mu_{t=1}=\mu_1, \\
    \partial_t\mu_t+\nabla\cdot(\mu_tv_t)=0.
\end{array}\right.
\end{equation}
The infimum is achieved by the couple $(\mu_t, \nabla\theta_t)_{0\leq t\leq 1}\in C^\infty(\XX)$, where $(\mu_t)_{t\in[0,1]}$ is the entropic interpolation between $\mu_0$ and $\mu_1$ and $(\nabla\theta_t)_{t\in[0,1]}$ appears in~\eqref{transport}.
\end{ethm}

\section{Main results}\label{sec-results}

The first result in this section is about the concavity of the exponential entropy along the entropic interpolation. For simplicity we state and prove our results in the Euclidian space $\XX$, but it is still true in more general cases, like a smooth complete connected Riemannian manifold satisfying the $CD(0,n)$ condition, under lighter assumptions on the marginals $\mu_0,\mu_1$. This result is a generalization of an older result known in information theory as the  Costa's Theorem~\cite{Cos85, CT}, that establishes the concavity of the exponential entropy along the heat flow.  
\begin{ethm}[Concavity of exponential entropy]\label{costa}
Let $R \in \PO$ be the reversible Brownian motion, $\mu_0,\mu_1 \in C_c^\infty(\XX)$. Let $(\mu_s)_{s\in[0,1]}$ be the entropic interpolation between the probability measures $\mu_0$ and $\mu_1$. Then the function, 
$$
\Psi : [0,1] \ni s \mapsto e^{-H(\mu_s|\dx)/n}
$$
is concave. 
\end{ethm}

\begin{proof}
The assumptions on the marginals $\mu_0, \mu_1$ make sure that the function $\Psi$ is smooth for all $0\leq s\leq 1$. Thus, to prove the concavity we will show that the second derivative of $\Psi$ is non positive. In the rest of the proof we use the shortest notation $\frac{d}{ds}H(\mu_s|\dx):=h'(s)$. \\ A double differentiation  provides,

\begin{eqnarray*}
\Psi''(s)&=&
\frac{e^{-H(\mu_s)/n}}{n}\left[\frac{1}{n}h'(s)^2-h''(s) \right].
\end{eqnarray*}
It remains to prove that 
\begin{equation}\label{toprove}
CD(0,n) \Longrightarrow \frac{1}{n}h'(s)^2-h''(s) \leq 0.
\end{equation}
By definition of the entropy functional, and the transport equation~\eqref{transport}, easy computations show that the first and second order  derivatives of the entropy along the entropic interpolation write as, (see~\cite{Leo12d}) 

\begin{eqnarray}
h'(s) &=& \int \nabla \theta_s\cdot\nabla\mu_s \,dx \label{der1}\\
h''(s) &=& \int \left[4\Gamma_2(\theta_s)+ \Gamma_2(\log\mu_s)\right] \;d\mu_s\label{der2}.
\end{eqnarray}
But the $CD(0,n)$ condition implies~\eqref{cd}, therefore,

\begin{eqnarray*}
h''(s) 
&\geq& \frac{1}{n} \int  (\Delta\theta_s)^2 d\mu_s\\
&\stackrel{(i)}{\geq}& \frac{1}{n} \left( \int \Delta\theta_s d\mu_s \right)^2\\
&\stackrel{(ii)}{=}& \frac{1}{n}\left(\int \nabla\theta_s \cdot \nabla \mu_s \,dx \right)^2=\frac{1}{n} \left(h'(s) \right)^2
\end{eqnarray*}
where $(i)$ follows from the Jensen's inequality and $(ii)$ from integration by parts \eqref{ipp}. Thus, it yields~\eqref{toprove} and this completes the proof.
\end{proof}

\begin{erem}\label{rem-costa}
Note that formally, if $\theta_s=-\log\mu_s/2$ in \eqref{transport} (hence in \eqref{der1} and~\eqref{der2}), then $\mu_s$ would be the heat flow and we would recover the stronger result obtained by Costa, namely, the concavity of the function $\Psi^2$. Though, when $\mu_s$ is the McCann interpolation, concavity of $\Psi$ is the best we can obtain, being equivalent to $CD(0,n)$ as proved in~\cite{EKS}. Our result shows that it is still true for the entropic approximation of the McCann geodesics, namely the entropic interpolations.
\end{erem}

\begin{erem}
An analogous result holds when we neglect the dimension and we add a drift to the generator, i.e. we consider $L=(\Delta -\nabla V\cdot\nabla)/2$, with reversing measure $dm=e^{-V}\dx$ where the potential $V$ is $\kappa$-convex for some non-negative $\kappa$. 
\end{erem}
The next theorem is about a differential property of the entropic cost. The stronger result holds under some more restrictive hypothesis on the marginal measures. To this aim let us introduce here the Schwartz space, that is the space of rapidly decreasing functions defined by, 
\begin{equation}\label{eq-schw}
\mathcal{S} = \{f : \XX\to \R, \;\textrm{s.t.}\; f\in C^\infty\, \textrm{and} \; \|x^\alpha D^\beta f\|_{\infty}<\infty, \forall \alpha, \beta\in \mathbb Z^n_+ \}.
\end{equation}
It is well known that $\mathcal S$ is closed under multiplication and convolution and for any $1\leq p\leq \infty$ $\mathcal S\subset L^p(\XX)$. \\
We consider all the couples of measures that admit a \emph{Schr\"odinger decomposition} of the type~\eqref{sch-system},  in which $f$ and $g$ are two positive functions in the Schwartz space \eqref{eq-schw} such that there exists some $\alpha>0$ such that $f,g\geq c_\alpha e^{-\alpha |x|^2}$ for some positive constant $c_\alpha$.  In other words, we define the set
 \begin{equation}
\Pi_{\mathcal{S}}= \left\{ (\mu_0,\mu_1) : \mu_0,\mu_1\in \PX \;\textrm{and}\; \exists \,f,g \;\textrm{s.t.}\;\left\{
\begin{array}{ll}   
\mu_0=fT_1g,\; \mu_1=gT_1f;\\
f,g \in \mathcal{S} \;\textrm{and}\; f,g>0;\\
\exists\, \alpha, c_\alpha>0 \;\textrm{s.t.}\; f,g\geq c_\alpha e^{-\alpha |x|^2}. \end{array}\right.
\right\} \label{hyp-pis}
\end{equation}
By the properties of the set \eqref{eq-schw}, it is immediate to see that if $(\mu_0,\mu_1)\in\Pi_{\mathcal S}$ then $\mu_0,\mu_1\in \Pi$ defined at \eqref{hyp-constr}. 

\begin{ethm}[Regularity of the entropic cost]\label{derivarive}
 Let $u, v$ be two probability densities with respect to the Lebesgue measure on $\XX$, and $(T_t)_{t\geq0}$ denote the heat semigroup. 
\begin{itemize} \item[(a)]
If $u,v$ are such that $u\dx, v\dx\in \Pi_{\mathcal{S}}$ as defined in~\eqref{hyp-pis}, then the function, 
$$
[0,\infty) \ni t \mapsto \A(u\,\dx,T_tv\,\dx) 
$$
is differentiable. In particular, for $t=0$ it holds, 
\begin{equation}\label{eq-thm}
\frac{d}{dt}\Big|_{t=0} \A(u,T_tv) = -\frac{1}{2}\frac{d}{ds}\Big|_{s=1} H(\mu_s|\dx)
\end{equation}
where $(\mu_s)_{0\leq s\leq 1}$ is the entropic interpolation between $u$ and $v$. \\
\item[(b)] If $u,v \in C_c^\infty(\XX)$, then, 
\begin{equation*}\label{eq-thm1}
\frac{d}{dt}^+\Big|_{t=0} \A(u,T_tv) \leq -\frac{1}{2}\frac{d}{ds}\Big|_{s=1} H(\mu_s|\dx).
\end{equation*}
Where, we used the notation 
\begin{equation*}\label{eq-upp-der}
\frac{d}{dt}^+f(t)=\limsup_{h\to0^+}\frac{f(t+h)-f(t)}{h}
\end{equation*}
to denote the super derivative.
\end{itemize}
\end{ethm}

\begin{proof}
We start by proving that
\begin{equation*}
\limsup_{t\to 0^+} \frac{\A(u,T_tv)-\A(u,v)}{t} \leq \\
-\frac{1}{2}\frac{d}{ds}\Big|_{s=1}H(\mu_s|\dx).
\end{equation*}
for $u,v\in C_c^\infty(\XX)$. This will prove $(b)$. The same arguments are valid under the assumption $(u,v) \in \Pi_{\mathcal S}$, thus we will complete the proof of $(a)$, by proving the converse inequality,  
\begin{equation}\label{first-ineq}
\liminf_{t\to0^+} \frac{\A(u, T_tv)-\A(u, \,v)}{t} \geq  -\frac{1}{2}\frac{d}{ds}\Big|_{s=1}H(\mu_s|\dx).
\end{equation}
Let $(\mu_s)_{s\in[0,1]}$ be the entropic interpolation between $u$ and $v$, with associated vector field $\nabla\theta_s$, verifying equation \eqref{transport}, that is, $\partial_s\mu_s+\nabla\cdot(\mu_s\nabla\theta_s)=0$. According to the Benamou-Brenier formulation at Theorem~\ref{bb} the entropic cost between $u$ and $v$ can be written as, 

\begin{equation}\label{bb-t0}
\A(u,v) =\frac{1}{2}\int_0^1\int_{\XX}\left(|\nabla\theta_s|^2+\frac{1}{4}|\nabla\log\mu_s|^2\right)\mu_s\,dzds.
\end{equation}
 Following a method already used in~\cite{DS08}, we define a new path between $u$ and $T_tv$, for all fixed $t\geq0$, by $$\displaystyle (\mu_s^t)_{s\in[0,1]}:=(T_{st}\mu_s)_{s\in[0,1]}.$$ Let us compute the derivative with respect to $s>0$, in order to see if it satisfies a transport equation,  
 \begin{eqnarray*}
 \partial_s \mu_s^t&=& \partial_s T_{st}\mu_s \\
 &=& \frac{t}{2}\Delta \mu_s^t-T_{st}\nabla\cdot(\mu_s\nabla\theta_s)\\
 &=&  \nabla \cdot\left(\frac{t}{2}\nabla \mu_s^t-\mathbf{T}_{st}(\mu_s\nabla\theta_s)\right)\\
 &=& \nabla \cdot\left[\mu_s^t\left(\frac{t}{2}\nabla\log\mu_s^t-\frac{1}{\mu_s^t}\mathbf{T}_{st}(\mu_s\nabla\theta_s)\right)\right]
 \end{eqnarray*}
where $(\mathbf{T}_{t})_{t\geq0}$ is, roughly speaking, the heat semigroup acting on $\XX$-valued functions as a standard heat semigroup on each coordinate, 

$$\mathbf{T}_t
\begin{pmatrix} f_1  \\
\vdots\\
f_n
\end{pmatrix}= \begin{pmatrix}
T_tf_1\\
\vdots\\
T_tf_n
\end{pmatrix}
$$
and the associated generator is $\mathbf{\Delta}/2$, acting on $\XX$-valued functions in similar way,
$$ \mathbf{\Delta}/2
\begin{pmatrix} f_1  \\
\vdots\\
f_n
\end{pmatrix}= \begin{pmatrix}
\Delta/2 f_1\\
\vdots\\
\Delta/2 f_n
\end{pmatrix}
$$
Therefore $(\mu_s^t)_{s\in[0,1]}$ satisfies the transport equation $\partial_s\mu_s^t+\nabla\cdot(\mu_s^tv_s^t)=0$, with the vector field 
\begin{equation}\label{vst}
v_s^t=-\frac{t}{2}\nabla\log\mu_s^t+\frac{\mathbf{T}_{st}(\mu_s\nabla\theta_s)}{\mu_s^t}.
\end{equation}
Moreover, $\mu_0^t=u$ and $\mu_1^t=T_tv$ for all $t\geq0$, and of course, for any $s>0$, $(\mu_s^t)_{t\geq0}$ is a probability on $\XX$. Then the three conditions in~\eqref{hyp-bb} are satisfied, and, by the Benamou-Brenier formulation (Theorem~\ref{bb}), we can write,

\begin{equation}\label{bb-t}
\mathcal{A}(u\,,T_tv) \leq \int_0^1\int \left(\frac{|v_s^t|^2}{2}\mu_s^t+\frac{1}{8}|\nabla\log\mu_s^t|^2 \mu_s^t\right) dx ds.
\end{equation}
Taking the difference between \eqref{bb-t} and \eqref{bb-t0},
\begin{multline}\label{divide}
\A(u, T_tv) - \A(u,v) \leq \int_0^1\int \left(\frac{|v_s^t|^2}{2}\mu_s^t+\frac{1}{8}|\nabla\log\mu_s^t|^2 \mu_s^t\right) dx ds\\ -\int_0^1\int \left(\frac{|v_s|^2}{2}\mu_s+\frac{1}{8}|\nabla\log\mu_s|^2 \mu_s\right) dx ds.
\end{multline}
Here we denoted $v_s=v_s^0=\nabla\theta_s$, by definition~\eqref{vst}. And recall that the couple $(\mu_s,\nabla\theta_s)$ is optimal in Theorem~\ref{bb} when $\mu_0=u$ and $\mu_1=v$.
Dividing \eqref{divide} by $t>0$ and taking the superior limit for $t\to 0^+$, we obtain
\begin{multline*}
\limsup_{t\to 0^+} \frac{\A(u,T_tv)-\A(u,v)}{t} \leq \frac{d}{dt}\Big|_{t=0}\left[\int_0^1\int \frac{|v_s^t|^2}{2}\mu_s^t\,dx ds\right.\\+ \left.\frac{1}{8}\int_0^1\int|\nabla\log\mu_s^t|^2 \mu_s^t dx ds\right].\label{i+ii}  
\end{multline*}
Note that on the right hand side, the superior limit is actually a limit since $\mu_0, \mu_1 \in C_c^\infty(\XX)$. Indeed, all the three terms are differentiable. 
From Lemma~\ref{calculs} below we can conclude that, 
\begin{equation*}
\limsup_{t\to 0^+} \frac{\A(u,T_tv)-\A(u,v)}{t} \leq \\
-\frac{1}{2}\frac{d}{ds}\Big|_{s=1}H(\mu_s|\dx).
\end{equation*}
This concludes the proof of $(b)$.  These arguments are valid also for $(u,v)\in\Pi_{\mathcal S}$, hence to conclude the proof of $(a)$ let us show \eqref{first-ineq}. Note that since $(u,v) \in  \Pi_{\mathcal{S}}$ then $u, T_tv \in \Pi$, therefore the entropic cost $\A(u,T_t v)$ is still finite. 
We will use here the Kantorovich dual formulation~\eqref{dual} for the forward entropic cost. Indeed, 
\begin{equation*}
\A(u, T_tv)-\A(u,v)\geq \int \xi T_tv \,dx - \int Q_1\xi u\,dx - \int \psi v\,dx +\int Q_1\psi u\,dx+\frac{1}{2}[H(v|\dx)-H(T_tv|\dx)]
\end{equation*}
where $\xi\in C_b(\XX)$ is any bounded continuous function and $\psi$ is optimal in~\eqref{dual} when $\mu_0=u$ and $\mu_1=v$. By choosing $\xi=\psi$, we get

\begin{equation}\label{for}
\A(u, T_tv)-\A(u, \,v) \geq \int \psi (T_tv-v) \,dx+\frac{1}{2}[H(v|\dx)-H(T_tv|\dx)]
\end{equation}
Note that since $g$ is in the Schwartz space, $\psi=\log g$ is smooth but not bounded. Hence we are not actually allowed to take $\xi=\psi$ but we should approximate $\psi$ by a sequence of bounded and continuous functions  by standard arguments. For the sake of simplicity we avoid here the technical details and use the inexact shortcut $\xi=\psi$. Note also that the right hand side of \eqref{for} is bounded from below thanks to the third hypothesis on the functions $f,g$ in the definition of $\Pi_{\mathcal S}$, namely up to (bounded) constant factors,,
$$
\int \psi (T_tv-v) dx\geq -\int |x|^2 (T_tv-v)dx \geq -\infty.
$$
The same is false in the case $f,g\in C_c^\infty(\XX)$ for which $\int \psi T_tv dx=-\infty$.  \\
On the other hand, by symmetry 
\begin{equation}\label{symmetry}
\A(u, T_tv)-\A(u, \,v) =\A(T_tv,u)-\A(v,u).
\end{equation}
Again, by duality~\eqref{dual}, 
\begin{multline*}
\A(T_tv,u)-\A(v,u) \geq \int \xi u \,dx - \int Q_1\xi T_tv\,dx - \int \psi^* u\,dx +\int Q_1\psi^* T_tv\,dx-\frac{1}{2}[H(v|\dx)-H(T_tv|\dx)].
\end{multline*}
Here, as before, we choose $\xi=\psi^*$, where $\psi^*$ is optimal in~\eqref{dual} in the reverse case, when $\mu_0=v$ and $\mu_1=u$. In this case we get, 

\begin{equation}\label{back}
\A(T_tv,u)-\A(v,u) \geq -\int Q_1\psi^*(T_tv-v) \,dx-\frac{1}{2}[H(v|\dx)-H(T_tv|\dx)]
\end{equation}
Take the half-sum of~\eqref{for} and~\eqref{back} and recall relation \eqref{symmetry}, 
$$
\A(u, T_tv)-\A (u, \,v)\geq \frac{1}{2}\left(\int \psi (T_tv-v) \,dx -\int Q_1\psi^*(T_tv-v) \,dx\right).
$$
We divide both sides by $t$ and take the inferior limit for $t\to0^+$. By the assumptions on the marginals $\mu_0, \mu_1$, the $\liminf$ on the right hand side is actually a limit. By definition of the infinitesimal generator (that we recall being for any measurable bounded function $f$ as $Lf:=\lim_{t\to0}(T_tf-f)/t$ we obtain, 

$$
\liminf_{t\to0^+}\frac{\A(u, T_tv)-\A(u, \,v)}{t} \geq  \frac{1}{2}\int \psi L v \,dx - \frac{1}{2}\int Q_1\psi^* L v \,dx
$$
and after integration by parts \eqref{ipp}, it yields to

\begin{equation}\label{final}
\liminf_{t\to0^+} \frac{\A(u,T_tv)-\A(u,v)}{t} \geq -\frac{1}{2}\int\left(\frac{\nabla\psi- \nabla Q_1\psi^*}{2}\right)\cdot \nabla v \,dx.
\end{equation}
Note that since $\psi^*$ is optimal, $Q_1\psi^*$ coincides with the potential $\psi_t^*=\log T_{1-t}g^*=\log T_{1-t}f$ for $t=0$. Therefore $\nabla Q_1\psi^*=\nabla\psi^*_0=\overrightarrow{v}^*_0=\overleftarrow{v}_1=\nabla\varphi_1$, where the middle equality is given by the relation~\eqref{vel-for-back}. We can conclude that on the right hand side of~\eqref{final} we have, 
$$
\nabla\theta_1 = \frac{\nabla \psi_1-\nabla\varphi_1}{2}
$$
and this proves~\eqref{first-ineq}. \\
We have proven that the inferior and supremum limits are bounded from above and below by the same quantity, and it proves \eqref{eq-thm}. Moreover by the semigroup property, the differential property can be extended to any $t\geq 0$.
\end{proof}

\begin{elem}\label{calculs}
$$
\frac{d}{dt}\Big|_{t=0}\int_0^1\int\left( \frac{|v_s^t|^2}{2}+ \frac{1}{8} |\nabla\log\mu_s^t|^2\right)  \mu_s^t\,dx ds=  -\frac{1}{2}\frac{d}{ds}\Big|_{s=1}H(\mu_s|\dx).
$$

\end{elem}

\begin{proof}
We denote for simplicity, 
\begin{eqnarray*}
(i)&=& \frac{d}{dt}\Big|_{t=0}\int_0^1\int \frac{|v_s^t|^2}{2} \mu_s^t\,dx ds,\\
(ii)&=& \frac{d}{dt}\Big|_{t=0}\int_0^1\int\left( \frac{1}{8} |\nabla\log\mu_s^t|^2\right)  \mu_s^t\,dx ds.
\end{eqnarray*}
Let us compute separately, these two derivatives, 
 
 \begin{eqnarray*}
 (i) &=& \left.\int_0^1\int v_s^t\frac{d}{dt}v_s^t \mu_s^t+\frac{|v_s^t|^2}{2}\frac{d}{dt}\mu_s^t\,dx ds\right|_{t=0}\nonumber\\
 &=& \left.\int_0^1\int v_s^t\cdot \left[-\frac{1}{2}\nabla\log\mu_s^t-\frac{t}{2}\nabla\left(\frac{s\Delta\mu_s^t}{2\mu_s^t}\right)+\frac{1}{(\mu_s^t)^2}\left(s  \frac{\mathbf{\Delta}}{2}\mathbf{T}_{st}(\mu_s\nabla\theta_s)\mu_s^t\right.\right.\right.\nonumber\\
 && \left.\left.\left. -\frac{s}{2}\mathbf{T}_{st}(\mu_s\nabla\theta_s)\Delta\mu_s^t\right) \right]\mu_s^t+\frac{|v_s^t|^2}{4}s\Delta\mu_s^t \,dx ds\right|_{t=0}\nonumber\\
 &=&\int_0^1\int \nabla\theta_s\cdot\left[-\frac{1}{2}\nabla\log\mu_s+\frac{1}{(\mu_s)^2}\left(s\frac{\mathbf{\Delta}}{2}(\mu_s\nabla\theta_s)\mu_s-\frac{s}{2}\mu_s\nabla\theta_s\Delta\mu_s\right) \right]\mu_s+\nonumber\\
 & &+\frac{|\nabla\theta_s|^2}{4}s\Delta\mu_s \,dx ds\nonumber\\
 &=&\frac{1}{2}\int_0^1\int -\nabla\theta_s\cdot\nabla\mu_s + s\nabla\theta_s\cdot\mathbf{\Delta}(\mu_s\nabla\theta_s)-s\frac{|\nabla\theta_s|^2}{2}\Delta\mu_s \,dx ds \nonumber\\
 &=&\frac{1}{2} \int_0^1\int -\nabla\theta_s\cdot\nabla\mu_s + s\left[\nabla\Delta\theta_s\cdot\nabla\theta_s\mu_s-\frac{1}{2}\Delta(|\nabla\theta_s|^2)\mu_s\right]\,dx ds\nonumber\\
 &\stackrel{\eqref{der1}}{=}& \int_0^1 \left[-\frac{1}{2}\frac{d}{ds}H(\mu_s|\dx)-2s\int \Gamma_2(\theta_s)\mu_s dx\right] ds.\label{(ii)}
\end{eqnarray*}
On the other hand, 

\begin{eqnarray*}
(ii) &=& \left.\frac{1}{8}\int_0^1\int 2 \nabla\log\mu_s^t\cdot\nabla\left(\frac{s\Delta\mu_s^t}{2\mu_s^t}\right)\mu_s^t+\frac{1}{2}|\nabla\log\mu_s^t|^2s\Delta\mu_s^t\, dx ds \right|_{t=0}\nonumber\\
&=& \frac{1}{8}\int_0^1\int s \nabla\log\mu_s\cdot\nabla\left(\frac{\Delta\mu_s}{\mu_s}\right)\mu_s+\frac{1}{2}|\nabla\log\mu_s|^2s\Delta\mu_s\, dx ds\nonumber\\
&=& \frac{1}{8}\int_0^1\int s(\nabla\log\mu_s\cdot\nabla(\Delta\log\mu_s+|\nabla\log\mu_s|^2)+\frac{1}{2}\Delta|\nabla\log\mu_s|^2)\mu_s\,dx ds\nonumber\\
&=& \frac{1}{8}\int_0^1\int s \nabla\log\mu_s\cdot\nabla\Delta\log\mu_s\mu_s-\frac{1}{2}s\Delta|\nabla\log\mu_s|^2\mu_s\,dx ds\nonumber\\
&=& -\frac{1}{2}\int_0^1\int s\Gamma_2(\log\mu_s)\mu_s \,dx ds\label{(iii)}
\end{eqnarray*}
Taking the sum of $(i)$ and $(ii)$,  
\begin{eqnarray*}
(i)+(ii) &=& \int_0^1 \left[-\frac{1}{2}\frac{d}{ds}H(\mu_s|\dx)-2s\int \Gamma_2(\psi_s)\mu_s+\frac{1}{2}\Gamma_2(\log\mu_s) dx\right] ds\nonumber\\
& \stackrel{\eqref{der2}}{=}& \int_0^1-\frac{1}{2}\frac{d}{ds}H(\mu_s|\dx)-\frac{s}{2}\frac{d^2}{ds^2}H(\mu_s|\dx) ds\nonumber\\
&=&-\frac{1}{2}\int_0^1\frac{d}{ds}\left(s\frac{d}{ds}H(\mu_s|\dx)\right) ds\nonumber\\
&=& -\frac{1}{2}\frac{d}{ds}\Big|_{s=1}H(\mu_s|\dx).\label{h'}
\end{eqnarray*}
\end{proof}
The next theorem is an immediate consequence of Theorem~\ref{costa} and Theorem~\ref{derivarive}. It establishes the Evolution Variational Inequality for the entropic cost, under the $CD(0,n)$ condition.

\begin{ecor}[EVI for entropic cost]\label{evi}
Under the same hypothesis of Theorem~\ref{derivarive} $(a)$, 

\begin{equation}\label{evi-eq}
\frac{d}{d t}^+\mathcal A(u\,\dx,T_t v\,\dx) \leq \frac{n}{2}\left(1-e^{-\frac{1}{n}[H(u|\dx) - H(T_t v|\dx)]}\right)
\end{equation}
for any $t\geq 0$.
If $u,v$ satisfy the hypothesis of Theorem~\ref{derivarive} $(b)$, then the standard derivative on the left hand side is replaced by the super derivative. 
\end{ecor}

\begin{proof}
Without loss of generality, by the semigroup property, it is enough to prove \eqref{evi-eq} for $t=0$, i.e.
$$
\frac{d}{dt}\Big|_{t=0} \mathcal A(u\,\dx,T_t v\,\dx) \leq \frac{n}{2}\left(1-e^{-\frac{1}{n}[H(u|\dx) - H(v|\dx)]}\right).
$$
By Theorem~\ref{costa}, the concavity of the function $\Psi$ implies that 
 $$
 \Psi'(1)\leq \Psi(1)-\Psi(0).
 $$
Taking into account the definition of $\Psi$, it is equivalent to say
 \begin{eqnarray*}
 e^{-H(v|\dx)/n}\left(-\frac{1}{n}\frac{d}{ds}\Big|_{s=1}H(\mu_s|\dx) \right) \leq e^{-H(v|\dx)/n} - e^{-H(u|\dx)/n}
 \end{eqnarray*}
that after rearranging the terms gives,  
 \begin{equation}\label{h'g}
 -\frac{d}{ds}\Big|_{s=1}H(\mu_s|\dx) \leq n\left(1-e^{-\frac{1}{n}[H(u|\dx)-H(v|\dx)]}\right).
 \end{equation}
 We conclude by applying \eqref{h'g} to \eqref{eq-thm}, to obtain the claimed result,
 $$
\frac{d}{dt}\Big|_{t=0}\mathcal{A}(u\,\dx,T_tv\,\dx) \leq \frac{n}{2}\left(1-e^{-\frac{1}{n}[H(u|\dx)-H(v|\dx)]}\right).
 $$ 
\end{proof}

\begin{erem}
We show later at Corollary~\ref{cor-w} that the EVI inequality for the entropic cost provides an immediate and alternative proof of the EVI inequality for the Wasserstein distance under the $CD(0,n)$ condition.
\end{erem}
The evolution variational inequality has some nice consequence. The first one stated in the first Corollary below, is the contraction of the entropic cost along the heat flow. It is an improvement of (\cite[Thm.~6.6 (b)]{GLR}) where dimensional contraction with respect to two different time variables $t,s\geq0$ is shown.

\begin{ecor}[Contraction]\label{cor-contr}
Under the same hypothesis of Theorem~\ref{derivarive} $(a)$, it holds for any $t>0$,

\begin{equation}\label{contr}
\A(T_{\tau}u,T_{\tau}v) \leq \A(u,v)-n\int_0^{\tau} \sinh^2\left(\frac{H(T_tu|\dx)-H(T_tv|\dx)}{2n}\right) dt.  
\end{equation}
\end{ecor}

\begin{proof}
To derive contraction from EVI, we follow a standard strategy already used to deduce contraction in Wasserstein distance, see~\cite{AGS05}. \\
Inequality~\eqref{evi-eq} is true for any probability density $u$, thus it is true in particular for $T_su$ for some fixed $s\geq0$,

\begin{equation}\label{evi-st}
\frac{d}{d t} \mathcal A(T_su,T_t v) \leq \frac{n}{2}\left(1-e^{-\frac{1}{n}[H(T_su|\dx) - H(T_t v|\dx)]}\right).
\end{equation}
We mentioned in Remark~\ref{remark} (ii), that the entropic cost is symmetric with respect to the initial and final measures, thus, 

\begin{equation*}
\frac{d}{d t} \mathcal A(T_t v,u) \leq \frac{n}{2}\left(1-e^{-\frac{1}{n}[H(u|\dx) - H(T_t v|\dx)]}\right)
\end{equation*}
and if we switch $u$ and $v$, it yields to

\begin{equation*}
\frac{d}{d t} \mathcal A(T_t u,v) \leq \frac{n}{2}\left(1-e^{-\frac{1}{n}[H(v|\dx) - H(T_t u|\dx)]}\right).
\end{equation*}
Since $v$ is also arbitrary, we can replace it by $T_sv$ for some $s\geq0$ fixed,

\begin{equation}\label{evi-ts}
\frac{d}{d t} \mathcal A(T_t u,T_sv) \leq \frac{n}{2}\left(1-e^{-\frac{1}{n}[H(T_sv|\dx) - H(T_t u|\dx)]}\right).
\end{equation}
We take the sum of \eqref{evi-st} and \eqref{evi-ts}, 
\begin{multline*}
\frac{d}{d t} \mathcal A(T_su,T_t v) + \frac{d}{d t} \mathcal A(T_t u,T_sv)\leq 	
\frac{n}{2}\left(2-e^{-\frac{1}{n}[H(T_su|\dx) - H(T_t v|\dx)]}-e^{-\frac{1}{n}[H(T_sv|\dx) - H(T_t u|\dx)]}\right).
\end{multline*}
We take $s=t$ and integrate in t between 0 and $\tau$, for some $\tau> 0$, 

\begin{eqnarray*}
\A(T_\tau u,T_\tau v)-\A(u,v) &\leq& \frac{n}{2}\int_0^\tau 2 -\left(e^{-\frac{1}{n}[H(T_tu|\dx) - H(T_t v|\dx)]}+\right.\left.e^{\frac{1}{n}[H(T_tu|\dx) - H(T_t v|\dx)]}\right) \;dt\\
&=& n\int_0^\tau 1- \cosh\left(\frac{H(T_tu|\dx)-H(T_tv|\dx)}{n}\right) dt.
\end{eqnarray*}
Moreover, by recalling that $\sinh^2x=(\cosh(2x)-1)/2$, we can rewrite the contraction inequality as, 

\begin{eqnarray*}
\A(T_\tau u, T_\tau v)-\A(u,v) \leq -2n\int_0^\tau \sinh^2\left(\frac{H(T_tu|\dx)-H(T_tv|\dx)}{2n}\right) dt.
\end{eqnarray*}
\end{proof}

\begin{erem}\label{rem-contr-w}
Contraction for the entropic cost, implies the analogue dimensional contraction for the quadratic Wasserstein distance along the heat flow ~\cite{BGG, BGGK}.
\begin{equation}
\label{eq-contraction-W2}
W_2^2(T_\tau u,T_\tau v)\leq W_2^2(u,v)  -4n\int_0^\tau \sinh^2\left(\frac{H(T_tu|\dx)-H(T_tv|\dx)}{2n}\right) dt.
\end{equation}
It can be seen, by considering \eqref{contr} for the $\varepsilon$-entropic cost $\A^\varepsilon$ defined in \eqref{eps-ec}, take the limit for $\varepsilon\to 0$ and recall the convergence property ~\eqref{gamma-lim}.
\end{erem}
Finally, we show an \emph{integral} and equivalent form of \eqref{evi}. Through the integral form we will be able in Corollary~\ref{cor-w} to deduce the classical EVI for the Wasserstein distance. 

\begin{eprop}\label{prop-integral}
Under the same assumption as in Theorem~\ref{derivarive}, the following statements are equivalent:
\begin{itemize}
    \item[(i)] For any $u, v \in \Pi_{\mathcal{S}}$, and any $t\geq0$,
    \begin{equation}\label{evi-der}
\frac{d}{d t} \mathcal A(u,T_t v) \leq \frac{n}{2}\left(1-e^{-\frac{1}{n}[H(u|\dx) - H(T_t v|\dx)]}\right);
\end{equation}
\item[(ii)]  For any $u, v \in \Pi_{\mathcal{S}}$,
\begin{equation*}\label{evi-der-zero}
\frac{d}{d t}\Big|_{t=0} \mathcal A(u,T_t v) \leq \frac{n}{2}\left(1-e^{-\frac{1}{n}[H(u|\dx) - H(v|\dx)]}\right);
\end{equation*}
    \item[(iii)] For any $u, v \in \Pi_{\mathcal{S}}$, and any $t\geq0$,
    \begin{equation}\label{evi-int}
    \A(u,T_tv)-\A(u,v)\leq \frac{n}{2}t\left(1-e^{-\frac{1}{n}[H(u|\dx) - H(T_t v|\dx)]} \right).
    \end{equation}
\end{itemize}
\end{eprop}

\begin{proof}
    $(i)\Rightarrow(iii)$ We integrate \eqref{evi-der} with respect to $0\leq t \leq \tau$, for any fixed $\tau>0$. 
    Thus, 
    \begin{eqnarray*}
    \A(u,T_\tau  v)-\A(u,v)&\leq& \frac{n}{2}\int_0^{\tau}1-e^{-\frac{1}{n}[H(u|\dx) - H(T_t v|\dx)]}dt\\
    &\leq& \frac{n}{2}\left(1-e^{-\frac{1}{n}[H(u|\dx) - H(T_{\tau} v|\dx)]}\right)\int_0^{\tau}dt	\\
    &=&\frac{n}{2}\tau\left(1-e^{-\frac{1}{n}[H(u|\dx) - H(T_{\tau} v|\dx)]}\right)
    \end{eqnarray*}
where the second inequality is given by the fact that the entropy is decreasing along the heat flow. \\    
    $(iii)\Rightarrow(ii)$ We divide by $t$ both sides of \eqref{evi-int}, and take the limit for $t\to 0^+$. The limit exists on the left hand side by Theorem~\ref{derivarive}, and on the right hand side by continuity. Therefore we obtain, 
    $$
\frac{d}{d t}\Big|_{t=0} \mathcal A(u\,\dx,T_t v\,\dx) \leq \frac{n}{2}\left(1-e^{-\frac{1}{n}[H(u|\dx) - H(v|\dx)]}\right).
$$
$(ii)\Rightarrow(i)$ It is true by the semigrop property.
\end{proof} 
\newline

\begin{ecor}[EVI for the Wasserstein distance]\label{cor-w}
Under the same assumption of Theorem~\ref{derivarive}, EVI for the Wasserstein distance holds for any $t\geq 0$,
\begin{equation}\label{evi-w}
\frac{d^+}{dt}W^2_2(u,T_tv) \leq n(1-e^{-\frac{1}{n}[H(u|\dx) - H(T_t v|\dx)]}).
\end{equation}
\end{ecor}

\begin{proof}
Let us consider \eqref{evi-der} for the $\varepsilon$-entropic cost~\eqref{eps-ec}. As we just showed, it implies the integral form~\eqref{evi-int} for the $\varepsilon$-entropic cost,
 \begin{equation*}
    \A^\e(u,T_tv)-\A^\e(u,v)\leq \frac{n}{2}t\left(1-e^{-\frac{1}{n}[H(u|\dx) - H(T_t v|\dx)]} \right).
    \end{equation*}
Taking the limit for $\varepsilon\to0$, and recalling \eqref{gamma-lim} we obtain, for any $t>0$,
$$
W_2^2(u,T_tv)-W_2^2(u,v)\leq nt\left(1-e^{-\frac{1}{n}[H(u|\dx)-H(T_tv|\dx)]} \right).
$$
Finally, divide by $t$ both sides and take the $\limsup$ for $t\to0^+$, to get
$$
\left.\frac{d}{dt}^+W_2^2(u,T_tv)\right|_{t=0} \leq n\left(1-e^{-\frac{1}{n}[H(u|\dx) - H(v|\dx)]}\right).
$$
Again, by the semigroup property, it implies \eqref{evi-w} for any $t\geq 0^+$. Remark that at the limit $\e \to 0$ the differential property at Theorem \ref{derivarive} is no more satisfied, hence we cannot have better than the sup-derivative of the Wasserstein distance.
\end{proof}
\begin{erem}
The absence of the factor $1/2$ on the left hand side in \eqref{evi-w} and \eqref{eq-contraction-W2}, is due to the fact that we have chosen the heat semigroup associated to $L=\Delta/2$, that is a more natural choise in the framework of the Schr\"odinger problem. It is straighforward to see that they are equivalent to their analogues in \cite{EKS, BGG} and \cite{BGGK}, where $L=\Delta$. 
\end{erem}
\textbf{Acknowledgments}\\
\thanks{I am grateful to my advisors,  Ivan Gentil and Christian Léonard, for the time and support  they devoted to me to conceive and improve this work. \\
This work was performed within the framework of the LABEX MILYON (ANR-10-LABX-0070) of Université de Lyon, within the program "Investissements d'Avenir" (ANR-11-IDEX-0007) operated by the French National Research Agency (ANR) and partially supported by the French ANR-12-BS01-0019 STAB project.\\
Part of this research has been performed at the Institut Mittag-Leffler in Stockholm during the semester \emph{Interactions between Partial Differential Equations and Functional Inequalities}. I am grateful to the institution and the organizers of the program for the hospitality and the excellent work environment.}

\newcommand{\etalchar}[1]{$^{#1}$}


\end{document}